\documentclass{daj}
\usepackage{amssymb, amsmath, amsfonts, amsthm, graphics}
\usepackage[hmargin=1 in, vmargin = 1 in]{geometry}
\usepackage{hyperref}
\usepackage{cleveref}
\usepackage{bm}

\dajAUTHORdetails{%
  title = {The growth rate of tri-colored sum-free sets}, %% please capitalize all significant words
  author = {Robert Kleinberg, Will Sawin, and David E Speyer},
  plaintextauthor = {Robert Kleinberg, Will Sawin, and David E Speyer},
}   %%% END \dajAUTHORdetails

\dajEDITORdetails{%
   year={2018},
   volume={XX},
   number={12},
   received={17 May 2017},   % received date: example: 7 January 2017
   revised={6 June 2018},    % Optional revised date (you may comment it out)
   published={6 July 2018},  % published date
   doi={10.19086/da.3734},       % XXX = number of paper, e.g. da006 for paper#6
}   %%% END \dajEDITORdetails

\newcommand{\newword}[1]{\textbf{\emph{#1}}}

\newcommand{\EE}{\mathbb{E}}
\newcommand{\FF}{\mathbb{F}}

\newcommand{\RR}{\mathbb{R}}

\newcommand{\ZZ}{\mathbb{Z}}

\newcommand{\vctr}[1]{{\bm{#1}}}

\newcommand{\hash}{{h}}

\newtheorem{theorem}{Theorem}

\newtheorem{lemma}[theorem]{Lemma}

\theoremstyle{definition}

\newtheorem{remark}[theorem]{Remark}

\begin{document}

\begin{frontmatter}[classification=text]
%% EDITOR: this will force the keywords to appear right after the Abstract.
%%   If the abstract is too long and would force the keywords off the
%%   front page, please comment out % [classification=text] above
%%   This way the keywords will be floated on the bottom of the first page
%%   even though the Abstract spills over to the next page.

%%% AUTHOR: Title goes here.  This line is optional.  You must use it
%%   if title has footnote attached or requires nontrivial typesetting,
%%   e.g., inclusion of linebreaks to force nice layout.
% \title{The growth rate of Tri-colored Sum-free sets} %% please capitalize all significant words

%%% AUTHOR:
%%% List all authors. If you wish, place grant acknowledgements in \thanks.
%%% In brackets include a short tag for each author.
\author[rk]{Robert Kleinberg }
\author[ds]{David E Speyer}
\author[ws]{Will Sawin}

%%% AUTHOR: Abstract goes here
\begin{abstract}
Let $G$ be an abelian group.  A tri-colored sum-free set in $G$ is a collection of triples $(\vctr{a}_i, \vctr{b}_i, \vctr{c}_i)$ in  $G$ such that $\vctr{a}_i+\vctr{b}_j+\vctr{c}_k=0$ if and only if $i=j=k$. Fix a prime $q$ and let $C_q$ be the cyclic group of order $q$. Let $\theta = \min_{\rho>0} (1+\rho+\cdots + \rho^{q-1}) \rho^{-(q-1)/3}$.  Blasiak, Church, Cohn, Grochow, Naslund, Sawin, and Umans (building on previous work of Croot, Lev and Pach, and of Ellenberg and Gijswijt) showed that a tri-colored sum-free set in $C_q^n$ has size at most $3 \theta^n$. Between this paper and a paper of Pebody, we will show that, for any $\delta > 0$, and $n$ sufficiently large, there are tri-colored sum-free sets in $C_q^n$ of size $(\theta-\delta)^n$. Our construction also works when $q$ is not prime. 
\end{abstract}
\end{frontmatter}

\section{Introduction}

Let $G$ be an abelian group. Let $\vctr{t} \in G^n$. We make the following slightly nonstandard definition: a \newword{sum-free set in $G^n$ with target $\vctr{t}$} is a  collection of triples $(\vctr{a}_i, \vctr{b}_i, \vctr{c}_i)$ in  $G^n \times G^n \times G^n$ such that $\vctr{a}_i+\vctr{b}_j+\vctr{c}_k=\vctr{t}$ if and only if $i=j=k$.
We may always replace $(\vctr{a}_i, \vctr{b}_i, \vctr{c}_i)$ by $(\vctr{a}_i, \vctr{b}_i, \vctr{c}_i - \vctr{t})$ to make the target $\vctr{0}$ (as we did in the abstract, and as is more standard), but allowing an arbitrary target will simplify our notation.
The usual terminology is ``tri-colored sum-free set", but we omit the reference to the coloring as we never consider any other kind.

If $X \subset G^n$ is a set with no three-term arithmetic progressions, then $\{ (\vctr{x}, \vctr{x}, -2 \vctr{x}) : \vctr{x} \in X \}$ is sum-free with target $\vctr{0}$, so lower bounds on sets without  three-term arithmetic progressions are also bounds on sum-free sets. The reverse does not hold: the largest known three-term arithmetic progression free subsets of $C_3^n$ (where $C_q$ is the cyclic group of order $q$) are of size $2.217^n$~\cite{Edel}. Before this paper, the largest known sum-free sets in $C_3^n$ were of size $2.519^n$~\cite{Alon-Shpilka-Umans}; this paper will raise the bound to $2.755^n$ and show that this bound is tight. 

Letting $r_3(G^n)$ denote the largest subset of $G^n$ with no three-term arithmetic progressions, the question of 
whether $\lim \sup_{n \to \infty} r_3(G^n)^{1/n} < |G|$ was open, until recently, for every abelian $G$ containing elements of order greater than two.
The breakthrough work of Croot, Lev, and Pach~\cite{CrootLP} introduced a polynomial method to 
prove that strict inequality holds when $G$ is cyclic of order 4,
and Ellenberg and Gijswijt~\cite{EllenbergG} built upon their ideas to prove it for cyclic groups of odd prime order.
Blasiak et al.~\cite{BCCGNSU} applied the same method to prove upper bounds for sum-free sets in $G^n$ for any fixed finite abelian group $G$.

We recall here one case of their bound. Let $C_q$ be the cyclic group of order $q$. Let $\theta = \min_{\beta>0} (1+\beta+\cdots + \beta^{q-1}) \beta^{-(q-1)/3}$ and let $\rho$ be the value of $\beta$ at which the minimum is attained. We note that the minimum is attained at a unique point which belongs to $(0,1)$ because $(1+\beta+\cdots + \beta^{q-1}) \beta^{-(q-1)/3}$ approaches $\infty$ as $\beta$ goes to $0$ from above, is increasing on the interval $[1,\infty)$, and has increasing first derivative on the interval $(0,1)$. 
% We note that $(1+\beta+\cdots + \beta^{q-1}) \beta^{-(q-1)/3}$ is convex when $q > 2$ and approaches $\infty$ as $\beta$ goes to $0$ or $\infty$, so the minimum is attained at a unique point. 

The following result of~\cite{BCCGNSU} is closely related to the results of~\cite{EllenbergG} (for primes) and~\cite[Theorem 4]{Petrov} (for prime powers). (What we denote $\theta$ is called $q J(q)$ in~\cite{BCCGNSU}.)
 
\begin{theorem}[{\cite[Theorem~4.14]{BCCGNSU}}] \label{UpperBound}
If $q$ is a prime power, then sum-free sets in $C_q^n$ have size at most $3 \theta^n$. 
\end{theorem}

%Therefore, sum-free sets in $C_q^n$ are of size $O(\theta^n)$ as $n \to \infty$ with $q$ fixed. 

Prior to this paper, it was not clear whether any of these applications of the polynomial method yielded tight
bounds. In fact, Theorem \ref{UpperBound} is tight to within a subexponential factor. 

 %polynomial method of Croot, Lev, and Pach~\cite{CrootLP} (in a ``symmetric
%formulation'' due to Tao~\cite{TaoBlog}) indeed yields tight bounds for sum-free sets in $G^n$ when $G$ is
%cyclic of prime order; 

\begin{theorem} \label{LowerBound}
Fix an integer $q \geq 2$. Define $\theta$ as above. For $n$ sufficiently large, there are sum-free sets in $C_q^n$ with size $\geq  \theta^n  e^{- 2 \sqrt{(2  \log 2 \log \theta ) n} - O_q(\log n)}$.
\end{theorem}

In this paper, we show Theorem \ref{LowerBound} except for a hypothesis on the existence of a probability distribution satisfying certain properties (Theorem \ref{DistributionExists}). In~\cite{Pebody}, Pebody will verify Theorem \ref{DistributionExists}, completing the proof of Theorem \ref{LowerBound}.\footnote{\cite{Pebody} was written in response to a preprint version of the paper which stated Theorem \ref{DistributionExists} as a conjecture, and we have chosen to preserve the chronology here, though otherwise updating the paper to reflect his result.}

The question of whether Theorem \ref{UpperBound} also yields a tight bound for $\lim \sup r_3(G^n)^{1/n}$ remains open.

%Our application of the polynomial method deviates slightly from the application
%by Ellenberg and Gijswijt~\cite{EllenbergG}, but for small values of $|G|$ we are able to verify numerically
%that the two methods yield the same bound; see Section~\ref{LinearProgramming} for further discussion of
%this point.

Sum-free sets  have applications in theoretical
computer science, especially the circle of ideas surrounding fast matrix multiplication algorithms. The
$O(n^{2.41})$ algorithm of Coppersmith and Winograd~\cite{Coppersmith} rests on a combinatorial
construction that can, in hindsight, be interpreted\footnote{This interpretation was made explicit by Fu and
Kleinberg~\cite{Fu}.} as a large sum-free set in $\FF_2^n$. In the same paper they presented a
conjecture in additive combinatorics that, if true, would imply that 
the exponent of matrix multiplication is 2, i.e., that there exist
matrix multiplication algorithms with running time $O(n^{2+\epsilon})$
for any $\epsilon>0$. This conjecture, along with another conjecture by
Cohn et al.~\cite{CohnKSU} that also implies the exponent of matrix
multiplication is 2, was shown 
by Alon, Shpilka, and Umans~\cite{Alon-Shpilka-Umans}
to necessitate the existence of sum-free 
sets of size $3^{n-o(n)}$ in $\FF_3^n$. The upper bound on sum-free
sets by Blasiak et al.~\cite{BCCGNSU} thus refutes both of these conjectures.
Furthermore, Blasiak et al.~\cite{BCCGNSU} show that a more general family 
of proposed fast matrix multiplication algorithms based on the ``simultaneous
triple product property'' (STPP)~\cite{CohnKSU} in an abelian group $H$ necessitates
the existence of sum-free sets of size $|H|^{1-o(1)}$. Their upper bound
on sum-free sets in abelian groups of bounded exponent thus precludes
achieving matrix multiplication exponent 2 using STPP constructions in such
groups. 

A second application of sum-free sets in theoretical computer science concerns 
property testing, the study of randomized algorithms for distinguishing functions 
$f$ having a specified property from those which have large Hamming
distance from every function that satisfies the property. A famous example
is the Blum-Luby-Rubinfeld (BLR) linearity tester~\cite{BlumLR}, which
queries the function value at only $O(1/\delta)$ points and
succeeds, with error probability less than $1/3$, in 
distinguishing linear functions
on $\FF_2^n$ from those that have distance $\delta \cdot 2^n$ from any linear function.
Testers which can distinguish low-degree polynomials on $\FF_2^n$
from those that are far from any low-degree polynomial constitute an 
important ingredient in the celebrated PCP Theorem~\cite{ALMSS}.
Bhattacharya and Xie~\cite{BX} demonstrated that constructions
of large sum-free sets in $\FF_2^n$ could be used to derive lower
bounds on the complexity of testing certain linear-invariant properties
of Boolean functions. 

Finally, sum-free sets have applications to removal lemmas in 
additive combinatorics, a topic that is heavily intertwined with property testing.
In particular, Green~\cite{Green05} proved an ``arithmetic removal lemma'' for
abelian groups which implies that for every $\epsilon>0$, there
is a $\delta>0$ such that for any abelian group $G$ and three subsets
$A, B, C$, either there are at least $\delta |G|^2$ distinct triples
$(a,b,c) \in A \times B \times C$ satisfying $a+b+c=0$, or one can
eliminate all such triples by deleting at most $\epsilon |G|$ elements
from each of $A,B,$ and $C$. Green's argument yields an upper bound
for $\delta^{-1}$ which is a tower of twos of height polynomial in 
$\epsilon^{-1}$. This bound can be improved using
combinatorial\footnote{See~\cite{FoxTR}, building upon the 
combinatorial proof of Green's result in~\cite{KralSV}.}
or Fourier analytic\footnote{See~\cite{HatamiST}, which pertains to the
case $G=\FF_2^n$ and adapts the
proof idea of~\cite{FoxTR} to the analytic setting.} techniques,
but for general abelian groups $G$ the value of $\delta$ is not bounded
below by any polynomial function of $\epsilon$. However,
when $G$ is the group $\FF_q^n$, Fox and Lovasz~\cite{FoxL} have applied our 
nearly-tight construction of sum-free sets in $G$ to obtain bounds of the form 
$$
  \epsilon^{-C_q + o(1)} \, < \, \delta^{-1} \, <(\epsilon/3)^{-C_q} ,
$$
where $C_q$ is a constant depending on $q$ but not $n$, and where $o(1)$ goes to $0$ as $\epsilon$ goes to $0$ for any fixed $q$.
%, related to the constant $\gamma$ that we define below.

\section{Notation}

Throughout this paper, we will use the following conventions: Lower case Roman letters denote integers, elements of cyclic groups (denoted $C_q$), of finite fields (denoted $\FF_q$), or general finite sets. Lower case Roman letters % marked with arrows 
in boldface denote elements of $\ZZ_{\geq 0}^m$  (for any $m$), $C_q^m$ or $\FF_q^m$. Capital Roman letters denote subsets of $\ZZ_{\geq 0}^m$,  $C_q^m$ or $\FF_q^m$. Lower case Greek letters denote real numbers; lower case Greek letters 
% marked with arrows 
in boldface denote elements of $\RR^m$. 
%capital Greek letters denote subsets of $\RR^m$. 
A notation such as $\alpha(x)$ or $\vctr{\alpha}(x)$ refers to a function of $x$ valued in real numbers, or real vectors.
For any sets $U$ and $V$, we write $U^V$ for the set of $U$-valued functions on $V$. All logarithms are to base $e$.

We fix a positive integer $q$. In section $4$, we will fix $n$ to be a positive integer divisible by $3$. The notation $O_q( \ )$ will always refer to bounds as $n \to \infty$ through integers divisible by $3$, with $q$ fixed. Let $\vctr{t} = (q-1, q-1, \ldots, q-1) \in \ZZ_{\geq 0}^n$.

We define the following sets of lattice points:
\[ \begin{array}{rcl}
I &= & \{ 0,1,\ldots, q-1 \} \subset \ZZ_{\geq 0} \\
T &=& \{ (a,b,c) \in I^3 : a+b+c = q-1 \} \\
%H &=& \{ (a_1, \ldots, a_n) \in I^n : \sum a_i = (q-1)n/3 \} \\
%Y &=& \{ (\vctr{a}, \vctr{b}, \vctr{c}) \in H^3  : \vctr{a}+\vctr{b}+\vctr{c} = \vctr{t} \} \\
\end{array} \]

\section{Entropy}

Let $A$ be a finite set and let $\vctr{e} = (e_1, e_2, \ldots, e_n) \in A^n$. We define the probability distribution $\vctr{\sigma}(\vctr{e})$ on $A$ by $\vctr{\sigma}_a(\vctr{e}) = \# \{ r : e_r = a \}/n$. In other words, $\vctr{\sigma}(\vctr{e})$ is the probability distribution of  uniformly randomly selecting an element of $\vctr{e}$. 

Let $A$ be a finite set and $\vctr{\lambda} \in \RR_{\geq 0}^A$ a probability distribution on $A$. The \newword{entropy}, $\eta(\vctr{\lambda})$, is defined by
\[ \eta(\vctr{\lambda}) = - \sum_{a \in A} \vctr{\lambda}_a \log(\vctr{\lambda}_a) \]
where $0 \log 0$ is considered to be $0$. 
The importance of the entropy function in our situation is the following:

\begin{lemma} \label{Histogram}
Let $A$ be a finite set, and let $\vctr{e}_0 \in A^n$.
Then 
\[
n \eta(\vctr{\sigma}(\vctr{e}_0)) - O_{|A|}(\log n)\leq \log \left( \# \left\{ \vctr{e} \in A^n : \vctr{\sigma}(\vctr{e}) = \vctr{\sigma}(\vctr{e}_0) \right\}  \right) \leq n \eta(\vctr{\sigma}(\vctr{e}_0)). \]
\end{lemma}

The implied constant in $O$ depends only on $|A|$ and not on $n$ or $\vctr{e}_0$.

\begin{proof}
For $a \in A$, let $n_a = n \vctr{\sigma}_a(\vctr{e}_0)$ be the number of times $a$ appears in $\vctr{e}_0$. 

The number of $\vctr{e} \in A^n$ such that $\vctr{\sigma}(\vctr{e}) = \vctr{\sigma}(\vctr{e}_0)$ is 
equal to the multinomial coefficient
\[
\binom{n}{(n_a)_{a \in A}} : =   \frac{n!}{\prod_{a \in A} n_a!}.
\]

For the upper bound, we take one term from the multinomial formula
\[ n^n = \left( \sum_{a\in A} n_a\right)^n \geq   \binom{n}{(n_a)_{a\in A}}  \prod_{a \in A} n_a^{n_a}, \]
so
\[  \binom{n}{(n_a)_{a\in A}} \leq \prod_{a \in A} \left(\frac{n}{n_a}\right)^{n_a}  = \exp(n \eta(\vctr{\sigma}(e_0))).\]

For the lower bound, we use the
following version of Stirling's formula.
(See, e.g.,~\cite{Robbins}.)
\[
    (n + \tfrac12) \log(n) - n + \tfrac12 \log(2 \pi) \;<\; \log(n!)  \;<\;  (n + \tfrac12) \log(n) - n + \tfrac12 \log(2 \pi)+\tfrac{1}{12}
\]
Applying this estimate to each of the factorial terms, and using $\sum_{a \in A} n_a = n$
we find that
\[
% \label{eq:Histogram}
  \left|
\log \binom{n}{(n_a)_{a \in A}}  - 
    \sum_{a \in A} n_a  \log \left( \frac{n}{n_a} \right) 
  \right|
  \leq
|A| \left[ \log(n) + \log(2 \pi) + \frac{1}{6} \right].
\]
Note that $\eta(\vctr{\sigma}(\vctr{e}_0)) = \sum_{a \in A} \frac{n_a}{n} \log \left( \frac{n}{n_a} \right)$, so this gives
\[
  \left|
\log \binom{n}{(n_a)_{a \in A}}  - 
    n \eta(\vctr{\sigma}(\vctr{e}_0)) 
  \right|
  \leq
|A| \left[ \log(n) + \log(2 \pi) + \frac{1}{6} \right]. \qedhere
\]
\end{proof}

If $A$ and $B$ are finite sets, $f: A \to B$ is a map and $\vctr{\lambda}$ is a probability distribution on $A$, then we define the probability distribution $f_{\ast} \vctr{\lambda}$ on $B$ by
\[ (f_{\ast} \vctr{\lambda})_b = \sum_{a \in f^{-1}(b)} \vctr{\lambda}_a . \]

It is well known that $\eta(f_{\ast} \vctr{\lambda}) \leq  \eta(\vctr{\lambda})$, with strict inequality if there are distinct elements $a_1$ and $a_2 \in A$ with $f(a_1) = f(a_2)$ and $\vctr{\lambda}_{a_1}$, $\vctr{\lambda}_{a_2} > 0$.

With $\rho$ and $\theta$ as defined before, 
define a probability distribution $\vctr{\psi}$ on $I$ by
\[ \vctr{\psi}_k = \frac{\rho^k}{1+\rho+\cdots + \rho^{q-1}}. \]
Let $f: T \to I$ be the map $f((i,j,k)) = k$. 
The following is proved in ~\cite{Pebody}.\footnote{A proof was also claimed in a preprint \cite{Norin}, but we are unable to confirm all the steps in the argument.}

%Differentiating with respect to $\sigma$, we have $\sum_{i=0}^{q-1} i \rho^i = \tfrac{q-1}{3} \sum_{i=0}^{q-1} \rho^i$. 
%Then the above formula shows that $\psi$ has expected value $\tfrac{q-1}{3}$. 

\begin{theorem}[{\cite[Theorem 4]{Pebody}}] \label{DistributionExists}
There is an $S_3$-symmetric probability distribution $\vctr{\pi}$ on $T$ with $f_{\ast} (\vctr{\pi}) = \vctr{\psi}$.
\end{theorem}

More precisely, \cite{Pebody} proves that $\vctr{\psi},\vctr{\psi},\vctr{\psi}$ are compatible in the sense that there are random variables $X_1,X_2,X_3$ whose distributions are each $\vctr{\psi}$ and such that $X_1+X_2+X_3$ is constant. As each variable has expectation $(p-1)/3$, that constant is certainly $p-1$, so $(X_1,X_2,X_3)$ is a random $T$-valued variable. Its probability distribution is a probability distribution on $T$ whose three projections are each $\vctr{\psi}$. Symmetrizing it, we obtain an $S_3$-symmetric probability distribution on $T$ whose projection under $f$ is $\vctr{\psi}$, as stated in Theorem \ref{DistributionExists}.

We will need to compute:
\begin{lemma} \label{EntropyOfPsi}
With notation as above, $\eta(\vctr{\psi}) = \log \theta$.
\end{lemma}

\begin{proof}
Note that
\[ \vctr{\psi}_k = \frac{\rho^{k-(q-1)/3}}{\theta}. \]

We have
\begin{equation} \label{eq:EntropyOfPsi.1}
   \eta(\vctr{\psi}) =-  \sum_{k \in I} \vctr{\psi}_k \log  \frac{\rho^{k-(q-1)/3}}{\theta}  = \left( \sum_{k \in I} \vctr{\psi}_k \right) \log \theta - \left( \sum_{k \in I} (k-(q-1)/3) \vctr{\psi}_k \right) \log \rho.
\end{equation}
The result follows by substituting 
\begin{align*} %\label{eq:EntropyOfPsi.2}
 & \sum_{k \in I} \vctr{\psi}_k = 1 \\
%\end{equation*}
%\begin{equation*}
%\label{eq:EntropyOfPsi.3}
 & \sum_{k \in I} (k-(q-1)/3) \vctr{\psi}_k = 
 \frac{\rho}{\theta} \cdot 
 \frac{d}{d \beta} \left[
   (1 + \beta + \cdots + \beta^{q-1}) \beta^{-(q-1)/3}
 \right]_{\beta = \rho} = 0,
\end{align*}
% and the result follows by 
% substituting~\eqref{eq:EntropyOfPsi.2}-\eqref{eq:EntropyOfPsi.3}
into~\eqref{eq:EntropyOfPsi.1}.
\end{proof}

\begin{remark}
If $\vctr{\pi}$ is any $S_3$-symmetric probability distribution on $T$ then $f_{\ast} (\vctr{\pi})$ has expected value $\tfrac{q-1}{3}$. Of all  probability distributions on $I$ with expected value $\tfrac{q-1}{3}$, the distribution $\vctr{\psi}$ has the greatest entropy. 
\end{remark}

\section{The construction}

Let $\vctr{\pi}$ be the probability distribution on $T$ guaranteed by Theorem~\ref{DistributionExists}. Fix $n$ divisible by $3$, so that when $S_3$ acts on the lattice $\ZZ^T$ by permuting the coordinates according to the $S_3$ action on $T$, the fixed point set of the action includes lattice vectors whose coordinates sum up to $n$. We can approximate $\vctr{\pi}$ to within $O_q(1/n)$ by an $S_3$-symmetric distribution $\vctr{\pi'}$ where the probability of each element is an integer multiple of $1/n$; such a $\vctr{\pi}'$ can be found by scaling down $\ZZ^T$ by $1/n$, taking the set of $S_3$-fixed points that belong to the probability simplex, and selecting the closest such  point to $\vctr{\pi}$. Then the marginal distribution $\vctr{\psi'}$ will be within $O_q(1/n)$ of $\vctr{\psi}$. The entropy function of a probability distribution, viewed as function of the vector of the probabilities of the elements, is a differentiable function on the open set of probability distributions assigning positive probability to every element. Thus, because $\vctr{\psi}$ assigns positive probability to each element, the entropy is Lipschitz in a neighborhood of $\vctr{\psi}$.  For large enough $n$, $\vctr{\psi'}$ is in that neighborhood, so 
\begin{equation} \label{EntropyOfPsiPrime}
  \eta(\vctr{\psi'})   = \eta(\vctr{\psi}) - O_q(1/n) = \log \theta - O_q(1/n) .
\end{equation} 
(The second equality is~\Cref{EntropyOfPsi}.)

%By the continuity of the entropy function, for sufficiently large $n$ 
%there exists a distribution $\vctr{\psi}_n$ obeying $\eta(\vctr{\psi'}) > \gamma - \frac12 \delta$ such that 
%$\vctr{\psi'} = \vctr{\sigma}(\vctr{a}) = \vctr{\sigma}(\vctr{b}) = \vctr{\sigma}(\vctr{c})$ 
%for some $(\vctr{a},\vctr{b},\vctr{c}) \in Y$.

Define the following sets:
\begin{align*}
  W &= \{ \vctr{a} \in I^n : \vctr{\sigma}(\vctr{a}) = \vctr{\psi'} \} \\
  V &= \{ (\vctr{a}, \vctr{b}, \vctr{c}) \in W^3 : \vctr{a}+\vctr{b}+\vctr{c} = \vctr{t} \} .
\end{align*}
We will show in Lemma~\ref{LowerBound0} that $|V|$ and $|W|$ grow exponentially in $n$, with $|V|$ having the faster growth rate. 
Our sum-free set in $C_q^n$ will be a subset of $V$.

Let $p$ be a prime number between $4 |V|/|W|$ and $8 |V|/|W|$ (such a prime exists by Bertrand's postulate). Since $|V|$ grows faster than $|W|$, the prime $p$ goes to $\infty$ as $n$ does. Let $S$ be a subset of $\FF_p$ having no three distinct elements in arithmetic
progression. 
Behrend's construction~\cite{Behrend}, with Elkin's improvement~\cite{Elkin}, implies that, for $p$ sufficiently large one can 
choose such a set whose cardinality is at least $p \cdot e^{-2\sqrt{2 \log 2 \log p}}$.

Let $\hash: \ZZ^{n+2} \to \FF_p$ be a linear map, chosen uniformly at random from all such linear maps.
For any $(\vctr{a},\vctr{b},\vctr{c}) \in V$, the sequence 
\[
  \hash(0,1,\vctr{a}), \;\; \tfrac12  \hash(1,1,\vctr{t}- \vctr{b}), \;\; \hash(1,0,\vctr{c})
\]
constitutes a (possibly degenerate) arithmetic progression in $\FF_p$. Thus,
this arithmetic progression is contained in $S$ if and only if its three terms are
all equal to one another and lie in $S$. Define $V'$ to be the subset of $V$ given by
\[
V' = {\Big  \{} (\vctr{a},\vctr{b},\vctr{c}) \in W^3 : \begin{array}{l} \vctr{a}+\vctr{b}+\vctr{c} = \vctr{t} \\  \hash(0,1,\vctr{a}) = \tfrac{1}{2}  \hash(1,1,\vctr{t}- \vctr{b}) =
\hash(1,0,\vctr{c}) \in S \end{array} {\Big \}} .
\]

Define 
$V''$ to be the set of all $(\vctr{a},\vctr{b},\vctr{c}) \in V'$ 
such that every other $(\vctr{a}',\vctr{b}',\vctr{c}') \in V'$ obeys
$\vctr{a'} \neq \vctr{a}, \vctr{b'} \neq \vctr{b}, \vctr{c'} \neq \vctr{c}$.

\begin{remark}
For this remark, assume $q$ is odd. Define a tri-colored $3$-AP-free set in $C_q^n$ to be a set of triples $(\vctr{a}_i, \vctr{b}'_i, \vctr{c}_i)$ in $(C_q^n)^3$ such that $\vctr{a}_i + \vctr{c}_k = 2 \vctr{b}'_j$ if and only if $i=j=k$. Replacing  $(\vctr{a}_i, \vctr{b}_i, \vctr{c}_i)$ with $(\vctr{a}_i, \tfrac{1}{2}(\vctr{t} - \vctr{b}) \bmod q, \vctr{c}_j)$ turns any  tri-colored sum-free set into a tri-colored $3$-AP-free set. 
In our set $V''$, each of $\vctr{a}$, $\vctr{b}$ and $\vctr{c}$ has entries distributed over $I$ with probability distribution $\vctr{\psi}$. Therefore in the tri-colored $3$-AP free set, the entries of $\vctr{a}$ and $\vctr{c}$ will be distributed with probability $\vctr{\psi}$, but the entries of $\vctr{b}$ will be distributed with the different distribution $g_{\ast} \vctr{\psi}$ where $g: I \to I$ is the map $g(b) = \tfrac{1}{2} (q-1-b) \bmod q$.  By contrast, if $X \subset C_q^n$ is a $3$-AP-free set in the standard sense, then $\{ (\vctr{x}, \vctr{x}, \vctr{x}) : \vctr{x} \in X \}$ is a  tri-colored $3$-AP-free set but, for this tri-colored $3$-AP-free set, each of the three components has the same distribution. This discrepancy suggests that it may be hard to lift our constructions out of the colored setting.
\end{remark}

The set $V''$ will be our sum-free set. We verify that it is sum-free in~\Cref{LowerBound1}.

\begin{lemma} \label{PsiExpectation}
For any $\vctr{a} = (a_1,a_2, \ldots, a_n) \in W$, we have $\sum a_i = n(q-1)/3$.
\end{lemma}

\begin{proof}
By definition,  $\vctr{\sigma}(\vctr{a})=\vctr{\psi}'$, so we want to show the expected value of the distribution $\vctr{\psi}'$ is $(q-1)/3$. But $\vctr{\psi}'$ is the marginal of the $S_3$ symmetric distribution $\vctr{\pi}'$ on $T$. As $\vctr{\pi'}$ is a symmetric distribution for a triple of random variables summing to $q-1$, the expectation of each variable must be $(q-1)/3$.
\end{proof}

\begin{lemma} \label{LowerBound1}
For any choice of the map $\hash$, the set $V''$ is a sum-free set
with target $\vctr{t}$ in $C_q^n$. 
\end{lemma}
\begin{proof}
Suppose that we have three (not necessarily distinct) triples 
$(\vctr{a}_i,\vctr{b}_i,\vctr{c}_i) \, (i=0,1,2)$ in $V''$ such that
$\vctr{a}_0 + \vctr{b}_1 + \vctr{c}_2 = \vctr{t}$ in $C_q^n$. 

We claim that we also have $\vctr{a}_0 + \vctr{b}_1 + \vctr{c}_2 = \vctr{t}$ in  $\ZZ^n$.
By~\Cref{PsiExpectation}, the entries of $\vctr{a}_0$, $\vctr{b}_1$ and $\vctr{c}_2$ each sum to $n(q-1)/3$ (in $\ZZ$) so  the sum of all the entries of $\vctr{a}_0 + \vctr{b}_1 + \vctr{c}_2$ (with the sum taken in $\ZZ$) must be $n(q-1)$.
Now the sum $\vctr{a}_0 + \vctr{b}_1 + \vctr{c}_2$ in $\ZZ^n$ has each entry congruent to $q-1$ mod $q$, by the assumption $\vctr{a}_0 + \vctr{b}_1 + \vctr{c}_2 = \vctr{t}$ in $C_q^n$, and each entry is nonnegative, because the entries of $\vctr{a}_0,\vctr{b}_1,$ and $\vctr{c}_2$ are nonnegative. So each entry is at least $q-1$. We just saw that the sum of all the entries is $n(q-1)$, so each entry is exactly $q-1$, as claimed.

Now that we know $\vctr{a}_0 + \vctr{b}_1 + \vctr{c}_2 = \vctr{t}$ in $\ZZ^n$, we deduce that 
$\left( \hash(0,1,\vctr{a}_0), \tfrac{1}{2} \hash(1,1,\vctr{t}-\vctr{b}_1), \hash(1,0,\vctr{c}_2) \right)$ is an arithmetic progression in $\FF_p$. 
Since $(\vctr{a}_0,\vctr{b}_0,\vctr{c}_0) \in V'$, we have $\vctr{a}_0 \in W$ and $\hash(0,1,\vctr{a}_0) \in S$. Similarly, 
$\vctr{b}_1, \, \vctr{c}_2 \in W$ and $\tfrac{1}{2} \hash(1,1,\vctr{t}-\vctr{b}_1), \, \hash(1,0,\vctr{c}_2) \in S$.
So $\left( \hash(0,1,\vctr{a}_0), \tfrac{1}{2} \hash(1,1,\vctr{t}-\vctr{b}_1), \hash(1,0,\vctr{c}_2) \right)$ is a (possibly degenerate) arithmetic progression in $S$.
As $S$ is arithmetic-progression-free, we must have $ \hash(0,1,\vctr{a}_0)= \tfrac{1}{2} \hash(1,1,\vctr{t}-\vctr{b}_1)=  \hash(1,0,\vctr{c}_2) \in S$. 
We have now checked that $(\vctr{a}_0, \vctr{b}_1, \vctr{c}_2)$ obeys all the conditions to be an element of $V'$.

Now, recalling the definition of $V''$ and the fact that $(\vctr{a}_i,\vctr{b}_i,\vctr{c}_i)\in V'$ for $i=0$, $1$, $2$,
we may conclude that $(\vctr{a}_i,\vctr{b}_i,\vctr{c}_i) = (\vctr{a}_0, \vctr{b}_1, \vctr{c}_2)$ for
$i=0$, $1$, $2$. In other words, the three triples $(\vctr{a}_0,\vctr{b}_0,\vctr{c}_0)$, 
$(\vctr{a}_1,\vctr{b}_1,\vctr{c}_1)$ and $(\vctr{a}_2,\vctr{b}_2,\vctr{c}_2)$ are
all equal to one another. 
\end{proof}

We will now begin to estimate the expected value of $|V''|$.

\begin{lemma} \label{LowerBound0}
We have
\[
  |V| \; \geq \; \exp(\eta(\vctr{\pi}') n - O_q(\log n)) 
\]
and
\[
  \exp(\eta(\vctr{\psi'}) n) \; \geq \; |W| 
  \; \geq \; \exp(\eta(\vctr{\psi}')n - O_q(\log n)) .
\]
\end{lemma}

Since $\vctr{\psi}' = f_{\ast} \vctr{\pi}'$, we have $\eta(\vctr{\pi}') \geq \eta(\vctr{\psi}')$. Moreover, if $n$ is large enough that the distribution $\vctr{\pi}'$ is not a point-mass on $(\frac{q-1}{3},\frac{q-1}{3},\frac{q-1}{3})$, then we have strict inequality since $\vctr{\pi}'$ is $S_3$-symmetric, so $\vctr{\pi}'_{ijk}>0$ implies  $\vctr{\pi}'_{jik}>0$. This establishes the previous claim that $|V|$ and $|W|$ grow exponentially, with $|V|$ having the faster rate.
\begin{proof}
Since $W= \{ \vctr{e} \in I^n : \vctr{\sigma}(\vctr{e}) = \vctr{\psi'} \} $, the lower and upper bounds for $|W|$ follow from \Cref{Histogram}.
We now need to establish the lower bound for $V$.

Let $V_0 = \{ \vctr{f} \in T^n : \vctr{\sigma}(\vctr{f}) = \vctr{\pi'} \}$. An element of $T^n$ is an $n$-tuple of triples of integers $((a_1, b_1, c_1), (a_2, b_2, c_2), \ldots, (a_n, b_n, c_n))$ with $a_i+b_i+c_i = q-1$. Reorganizing these integers as $((a_1, a_2, \ldots, a_n), (b_1, b_2, \ldots, b_n), (c_1, c_2, \ldots, c_n))$, we obtain a triple of length $n$ vectors $\vctr{a}$, $\vctr{b}$ and $\vctr{c}$ with $\vctr{a}+\vctr{b}+\vctr{c}= \vctr{t}$. Let us apply this construction to some $\vctr{f}$ in $V_0$ to get some $\vctr{a}$, $\vctr{b}$ and $\vctr{c}$. Since $\vctr{\pi}'$ is $S_3$ symmetric, we have $\vctr{\sigma}(\vctr{a}) = \vctr{\sigma}(\vctr{b}) = \vctr{\sigma}(\vctr{c}) = \vctr{\psi}'$ so $\vctr{a}$, $\vctr{b}$ and $\vctr{c}$ lie in $W$ and $(\vctr{a}, \vctr{b}, \vctr{c}) \in V$. This construction gives an injection from $V_0$ into $V$, so $|V| \geq |V_0|$.

By~\Cref{Histogram}, $|V_0| = \exp(\eta(\vctr{\pi}') n - O_q(\log n))$, so $|V| \geq  \exp(\eta(\vctr{\pi}') n - O_q(\log n))$ as desired.
\end{proof}

\begin{lemma} \label{LinIndep}
%Suppose $p$ is not a divisor of $(q-1)n/3$. 
Suppose $p > q$.
For any two distinct elements $(\vctr{a},\vctr{b},\vctr{c})$,
$(\vctr{a}',\vctr{b}',\vctr{c}') \in V$, the $(n+2) \times 6$-matrix over 
$\FF_p$ given by
\[
  M = \begin{pmatrix}
%    0 & \vctr{a} \\ 0 & \vctr{a}' \\ 1 & \vctr{b} \\
%    1 & \vctr{b}' \\ 1 & \vctr{t}-\vctr{c} \\ 1 & \vctr{t}-\vctr{c}'
  0 & 0 & 1/2 & 1/2 & 1 & 1 \\
1 & 1 & 1/2 & 1/2 & 0 & 0 \\
  \vctr{a} & \vctr{a}' & (\vctr{t}-\vctr{b})/2 & (\vctr{t}- \vctr{b}')/2 &
\vctr{c} &\vctr{c}' \\
  \end{pmatrix}
\]
has rank at least $3$.
\end{lemma}

\begin{proof}
The first two rows already have rank $2$, so we simply must show that the bottom $n$ rows are not all in the span of the first two. 
If the bottom $n$ rows were in the span of the first two, then modulo $p$ the first column would equal the second, the third column equal the fourth, and the fifth column equal the sixth. Since the entries of the matrix are between $0$ and $q-1$, and $p > q$, equality of columns modulo $p$ implies outright equality. This gives $\vctr{a} = \vctr{a}'$, $\vctr{b} = \vctr{b}'$ and $\vctr{c}=\vctr{c}'$, contrary to our assumption that  $(\vctr{a},\vctr{b},\vctr{c})$ and
$(\vctr{a}',\vctr{b}',\vctr{c}')$ are distinct.
\end{proof}

%Suppose $\vctr{a} - \vctr{a}' \neq 0$. Then we can choose $\vctr{e}$ such that
%$\transpose{(\vctr{a}-\vctr{a}')} \vctr{e} \neq 0$. Then we find that 
%\begin{align*}
%  \begin{pmatrix} 
%    0 & 0 & 1 & 0 & 0 & 0 \\
%    1 & 0 & 0 & 0 & 0 & 0 \\
%    1 & -1 & 0 & 0 & 0 & 0 
%  \end{pmatrix}
%  \transpose{M} 
%  \begin{pmatrix}
%    1 & 0 & 0 \\
%    \vctr{0} & \vctr{1} & \vctr{e}
%  \end{pmatrix}
%  &=
%  \begin{pmatrix}
%    1 & \transpose{\vctr{b}} \\
%    0 & \transpose{\vctr{a}} \\
%    0 & \transpose{(\vctr{a}-\vctr{a}')}
%  \end{pmatrix}
%  \begin{pmatrix}
%    1 & 0 & 0 \\
%    \vctr{0} & \vctr{1} & \vctr{e}
%  \end{pmatrix} \\
%  & = 
%  \begin{pmatrix}
%     1 & \ast & \ast \\
%     0 & (q-1)n/3 & \ast \\
%     0 & 0 & \transpose{(\vctr{a}-\vctr{a}')}
%  \end{pmatrix}
%\end{align*}
%By our assumption that $p$ is not a divisor of $(q-1)n/3$, the
%matrix on the right side has rank 3, so $M$ has rank at least 3.
%A similar calculation applies if $\vctr{b} - \vctr{b}' \neq 0$ or 
%if $\vctr{c} - \vctr{c}' \neq 0$.

%EDIT FROM HERE. This argument is missing the case a \neq a', b \neq b', c \neq c' ! We need to know the matrix is rank 4 in this case!

\begin{lemma} \label{LowerBound2}
%Suppose $p$ is not a divisor of $(q-1)n/3$. 
When $p >q $ and  $\hash$ is a uniformly random homomorphism of $\ZZ^{n+2}$ to $\FF_p$,
the expected cardinality of $V''$ is at least $\frac{1}{32}e^{-2\sqrt{2 \log 2 \log p}} \cdot |W|$.
\end{lemma}
\begin{proof}
For any $(\vctr{a},\vctr{b},\vctr{c}) \in V$, we want to compute the probability that
there exists $s \in S$ such that 
\begin{equation} \label{eq:LB2.1}
  \hash(0,1,\vctr{a}) = \tfrac{1}{2}  \hash(1,1, \vctr{t}-\vctr{b}) = \hash(1,0,\vctr{c}) = s.
\end{equation}
Furthermore, since $\hash(0,1,\vctr{a}), \, \frac12  \hash(1,1, \vctr{t}-\vctr{b}), \,
\hash(1,0,\vctr{c})$ always form a (possibly degenerate) arithmetic
progression, if any two of these values are equal to $s$ then the third one 
equals $s$ as well. The vectors $(0,1,\vctr{a})$ and $(1,0,\vctr{c})$ are 
linearly independent modulo $p$, so the pair 
$(\hash(0,1,\vctr{a}),   \hash(1,0,\vctr{c}))$ is 
uniformly distributed in $\FF_p^2$ and the probability
that~\eqref{eq:LB2.1} is satisfied for a fixed $s \in S$ is $p^{-2}$. Summing
over all $(\vctr{a},\vctr{b},\vctr{c}) \in V$ and $s \in S$ we obtain
\begin{equation} \label{eq:LB2.2}
  \EE ( |V'| ) = \frac{|V| |S| }{ p^2}.
\end{equation}
An element $(\vctr{a},\vctr{b},\vctr{c}) \in V'$ belongs to $V''$ unless
there exists some other $(\vctr{a}', \vctr{b}', \vctr{c}') \in V'$ such that
one of the equations $\vctr{a}=\vctr{a}', \, \vctr{b}=\vctr{b}'$, or
$\vctr{c}=\vctr{c}'$ holds. In order for any such equation to hold, it
must be the case that there is a single element $s \in S$ such that 
\begin{equation} \label{eq:LB2.3}
  s = \hash(0,1,\vctr{a}) = \hash(0,1,\vctr{a}') =
  \tfrac12 \hash(1,1,\vctr{t}-\vctr{b}) = \tfrac12 \hash(1,1,\vctr{t}-\vctr{b}') =
  \hash(1,0,\vctr{c}) = \hash(1,0,\vctr{c}').
\end{equation}
By \Cref{LinIndep}, the six-tuple $(\hash(0,1,\vctr{a}), \, \hash(0,1,\vctr{a}'), \,
  \tfrac12 \hash(1,1,\vctr{t}-\vctr{b}) , \, \tfrac12 \hash(1,1,\vctr{t}-\vctr{b}') , \,
  \hash(1,0,\vctr{c}) , \, \hash(1,0,\vctr{c}'))$
is uniformly distributed on a subspace of $\FF_p^6$ of dimension
at least 3. 
Hence, for any $(\vctr{a},\vctr{b},\vctr{c}),(\vctr{a}',\vctr{b}',\vctr{c}')\in V$ and for a fixed $s$, the probability that~\eqref{eq:LB2.3} holds is 
at most $p^{-3}$. 
The probability that there exists some $s$ for which~\eqref{eq:LB2.3} holds is thus bounded by $|S| p^{-3}$. 

For any $(\vctr{a},\vctr{b},\vctr{c}) \in V$, the number of elements $(\vctr{a}',\vctr{b}',\vctr{c}') \in V$ such that $\vctr{a}' = \vctr{a}$ is equal to $|V|/|W|$.
(To see this, note that the group $S_n$ acts on $V$ and $W$ by 
permuting the coordinates of vectors. These actions are compatible with
the projection map $V \to W$ defined by 
$(\vctr{a},\vctr{b},\vctr{c}) \mapsto \vctr{a}$.
The fibers of this projection map must be equinumerous
because the action of $S_n$ on $W$ is transitive.)
Thus, for any $(\vctr{a},\vctr{b},\vctr{c}) \in V$
the probability that $(\vctr{a},\vctr{b},\vctr{c})$ belongs 
to $V'$ but not $V''$
because it ``collides'' with another ordered triple of the form $(\vctr{a},\vctr{b}',\vctr{c}')$
in $V'$ is bounded above by $\frac{|V|}{|W|}  |S|   p^{-3}$. 
The analogous counting argument applies to collisions with triples of the form
$(\vctr{a}',\vctr{b},\vctr{c}')$ and $(\vctr{a}',\vctr{b}',\vctr{c})$.
Summing over $|V|$ choices of $(\vctr{a},\vctr{b},\vctr{c})$,
%and $|S|$ choices of $s \in \FF_p$, 
we find that the expected cardinality
of $V' \setminus V''$ is bounded above by 
\[
    3 |V| \frac{|V|}{|W|} |S|  p^{-3} =
    \frac{3 |V|}{p|W|} \cdot \frac{|V| |S|}{p^2} <
    \frac{3}{4} \cdot \EE ( |V'| ).
\]
Thus,
\[
   \EE( |V''|) \geq \frac14 \EE (|V'|) = \frac{|V| |S|}{4 p^2} = \frac14 \cdot \frac{|V|}{p} \cdot \frac{|S|}{p}
    > \frac{e^{-2\sqrt{2 \log 2 \log p}}}{32} \cdot  |W|.
\]
\end{proof}

We now prove our main theorem.

\begin{theorem} 
If $n$ is sufficiently large then 
there exists a sum-free set in $C_q^n$ with target $\vctr{t}$ 
whose size is greater than $\theta^n e^{ - 2 \sqrt{2\log 2 \log \theta \ n  }  - O_q(\log n)}$.
\end{theorem}
\begin{proof}
The random set $V''$ constructed above is a sum-free set in $C_q^n$
with target $\vctr{t}$ (\Cref{LowerBound1}) and its expected size
is greater than $\frac{1}{32} e^{-2\sqrt{2 \log 2 \log p}}\cdot |W|$ (\Cref{LowerBound2}), because we may take $n$ large enough that $p>q$. 
Using \Cref{LowerBound0} we have 
\[
  |W| \geq \exp(\eta(\vctr{\psi}') \, n - O_q(\log n) ) \geq \exp((\log \theta - O_q(1/n)) \, n - O_q(\log n)) \geq \theta^n \exp( - O_q(\log n))
\]
for all sufficiently large $n$. 
The inequality $|V| \leq |W|^2$ holds because 
the projection map $V \to W^2$ defined by 
$(\vctr{a},\vctr{b},\vctr{c}) \mapsto (\vctr{a},\vctr{b})$
is one-to-one. This justifies the second inequality in
\[
  p < 8 \frac{|V|}{|W|} \leq 8 |W| < 8 \exp(\eta(\vctr{\psi}') \, n)
    \leq 8 \exp ( (\log \theta + O_q(1/n) ) \, n ),
\]
while the third inequality follows from \Cref{LowerBound0}.
Taking logarithms of both sides, we deduce that 
$\log p < n \log \theta + O_q(1)$, and hence
\[
e^{- 2 \sqrt{2 \log 2 \log p} }> e^{- 2 \sqrt{2 \log 2 (n \log \theta + O_q(1))} }  > e^{- 2 \sqrt{2 \log 2  \log \theta \ n}  -O_q(1/\sqrt{n})} .
\]
Hence,
\[
  \EE (|V''|) > \frac{1}{32} e^{- 2 \sqrt{2 \log 2  \log \theta \ n}  -O_q(1/\sqrt{n})}
 \cdot |W| 
    \geq \theta^n e^{ - 2 \sqrt{2 \log 2  \log \theta \ n} - O_q(\log n)}
\]
for sufficiently large $n$. The theorem follows because there must exist at least one choice of 
$\hash$ for which the cardinality of the random set $V''$ is at least as large as its expected value.
\end{proof}

It follows from Roth's theorem that our construction produces sum-free sets $V'' \subseteq V$ of size $\EE(|V''|) \leq\EE(V') = \frac{V |S|}{p^2}=o(|W|)$ regardless of how we choose $S$. We do not know if an arbitrary sum-free set contained in $V$ must have size $o(|W|)$, only the trivial bound $ |W|$. It would be interesting to improve this situation.

\section*{Acknowledgements}

We would like to thank Henry Cohn,  Jacob Fox, L\'aszl\'o Mikl\'os Lov\'asz, and Terence Tao for helpful conversations.
We would particularly like to thank Jordan Ellenberg for a series of blog posts which drew our attention to this problem and sparked our collaboration. 
The first author was employed at Microsoft Research New England at the time these results were discovered, and he is grateful to Microsoft Corporation 
for their support of this research.
The second author was supported by Dr. Max R\"{o}ssler, the Walter Haefner Foundation and the ETH Zurich Foundation.
The third author was partially supported by NSF grant DMS-1600223.

%\raggedright
\bibliographystyle{amsplain}

\begin{dajauthors}
\begin{authorinfo}[rk]
Robert Kleinberg \\
Department of Computer Science \\ Cornell University \\ Ithaca, NY 14853, USA \\
robert\imagedot{}kleinberg\imageat{}cornell\imagedot{}edu \\
\end{authorinfo}
\begin{authorinfo}[ws]
Will Sawin \\
ETH Institute for Theoretical Studies \\ ETH Zurich \\ 8092 Z\"{u}rich, Switzerland \\
william\imagedot{}sawin\imageat{}math\imagedot{}ethz\imagedot{ch} \\
\end{authorinfo}
\begin{authorinfo}[ds]
David E Speyer \\
Department of Mathematics \\ University of Michigan \\ Ann Arbor, MI 48109, USA \\
speyer\imageat{}umich\imagedot{}edu \\
\end{authorinfo}
\end{dajauthors}

\end{document}